\documentclass{amsproc}
\usepackage{amssymb}
\usepackage{amsfonts}
\usepackage{lscape}
\usepackage{lscape,booktabs}

\theoremstyle{plain}

\numberwithin{equation}{section}

\begin{document}
\title[Smoothness of orbital measures]{Smoothness of convolution products of
orbital measures on rank one compact symmetric spaces}
\author{Kathryn E. Hare}
\address{Dept. of Pure Mathematics\\
University of Waterloo\\
Waterloo, Ont.,~Canada\\
N2L 3G1}
\email{kehare@uwaterloo.ca}
\author{Jimmy He}
\address{Dept. of Pure Mathematics\\
University of Waterloo\\
Waterloo, Ont.,~Canada\\
N2L 3G1}
\email{jimmy.he@uwaterloo.ca}
\thanks{This work was supported in part by NSERC Grant 2011-44597. }

\theoremstyle{cupthm}
\newtheorem{thm}{Theorem}[section]
\newtheorem{prop}[thm]{Proposition}
\newtheorem{cor}[thm]{Corollary}
\newtheorem{lemma}[thm]{Lemma}
\theoremstyle{cupdefn}
\newtheorem{defn}[thm]{Definition}
\theoremstyle{cuprem}
\newtheorem{rem}[thm]{Remark}


\begin{abstract}
We prove that all convolution products of pairs of continuous orbital
measures in rank one, compact symmetric spaces are absolutely continuous and
determine which convolution products are in $L^{2}$ (meaning, their density
function is in $L^{2})$. Characterizations of the pairs whose convolution product is either absolutely continuous or in $L^2$ are given in terms
of the dimensions of the corresponding double cosets. In particular, we
prove that if $G/K$ is not $SU(2)/SO(2),$ then the convolution of any two
regular orbital measures is in $L^{2}$, while in $SU(2)/SO(2)$ there are no
pairs of orbital measures whose convolution product is in $L^{2}$.

\end{abstract}


\maketitle

\section{Introduction}

Let $G/K$ be an irreducible, simple, simply connected, compact symmetric
space. By an orbital measure, $\mu _{z}$, we mean the $K$-bi-invariant,
singular measure on $G$ supported on the double coset $KzK$. In this note we
prove that in any rank one symmetric space the convolution product of two
orbital measures, $\mu _{z_{1}}\ast \mu _{z_{2}},$ is absolutely continuous
if and only if 
\begin{equation}
\dim Kz_{1}K+\dim Kz_{2}K\geq \dim G/K  \label{criteria}
\end{equation}
if and only if $\mu _{z_{1}}$ and $\mu _{z_{2}}$ are both continuous. For
short, we write $\mu _{z_{1}}\ast \mu _{z_{2}}\in $ $L^{1}(G)$ because being
absolutely continuous is equivalent to the density function belonging to 
$L^{1}$. Furthermore, we prove that $\mu _{z_{1}}\ast \mu _{z_{2}}\in L^{2}(G)$ if and only
if the inequality \eqref{criteria} is strict. We show that only four of the
infinitely many rank one symmetric spaces admit any pair of continuous
orbital measures whose convolution is not in $L^{2}$.

It was previously shown in \cite{AGP} that there are continuous orbital
measures in the rank one symmetric space $SU(2)/SO(2)$ whose convolution is in $L^1$, but
not in $L^{2}$. This came as a surprise because in the special case that the
symmetric space is $(H\times H)/H\sim H$ for a compact Lie group $H$, it is
known that $\mu _{z}^{p}\in L^{1}(H)$ if and only if $\mu _{z}^{p}\in
L^{2}(H)$ for all integers $p$, the exponent $p$ here meaning the $p$-fold
convolution product \cite{GHAdv}.  One consequence of our characterization is that it
follows that there are no pairs of orbital measures for $SU(2)/SO(2)$ whose
convolution is in $L^{2}$.

The continuous orbital measures on $SU(2)/SO(2)$ are all examples of what
are called `regular' orbital measures. (For the definition, see section 2.)
Previously, it was shown that in any symmetric space the convolution of two
regular orbital measures is in $L^{1}$ \cite{GHBAMS}. Here we see that in
any rank one symmetric space, the convolution of any two continuous orbital
measures is in $L^{1}$ and if $G/K$ is any rank one symmetric space other
than $SU(2)/SO(2)$, then the convolution of any two regular orbital measures
is in $L^{2}(G)$. We also prove that in any rank one symmetric space, the
product of any three continuous orbital measures belongs to $L^{2}$.
Previously it was known that such a $3$-fold product was in $L^{1}$ 
\cite{GHJMAA}, with the sharper $L^{2}$ result known only for $SU(2)/SO(2)$ 
\cite{AGP}.

The problem of establishing the absolute continuity of convolution products
of orbital measures was originally studied by Ragozin in \cite{Ra}.
Extensive treatment of the absolute continuity problem in the non-compact
case has been carried out by Graczyk and Sawyer, c.f. \cite{GSJFA}, 
\cite{GSLie}.

\section{Absolutely continuous convolution products}

\subsection{Notation and Terminology}

If $G$ is a compact group and $K$ a compact, connected subgroup fixed by an
involution $\theta $, then $G/K$ is called a compact symmetric space. We
will assume $G/K$ is an irreducible, simple, simply connected, compact
symmetric space of Cartan type I. Our primary interest are those of rank
one; see the appendix for a complete list. We let $\mathfrak{g}=\mathfrak{k}
+i\mathfrak{p}$ be the Cartan decomposition of $\mathfrak{g}$, the Lie
algebra of $G$, let $\mathfrak{a}$ denote a maximal abelian subalgebra of 
$\mathfrak{p}$ and assume $\mathfrak{t}$ is a torus of $\mathfrak{g}$ that
contains $\mathfrak{a}$. Then $K=exp(\mathfrak{k})$ and if we let $A=exp(i
\mathfrak{a}),$ we have $G=KAK$. Hence every double coset, $KzK$, contains
an element $z$ in $A$.

We denote by $\Sigma ^{+}$ the set of positive roots of $(\mathfrak{g,t})$
and let 
\begin{equation*}
\Phi ^{+}=\{\alpha =\beta |_{\mathfrak{a}}:\beta \in \Sigma ^{+},\beta |_{
\mathfrak{a}}\neq 0\}
\end{equation*}
be the set of (positive) \textit{restricted roots}. When $G/K$ is rank one,
there is either one positive restricted root, $\alpha $, or there are two, 
$\alpha$ and $2\alpha $. We write $m_{\beta }$ for the multiplicity of the
restricted root $\beta $, that is, the dimension of the restricted root
space $\mathfrak{g}_{\beta }$. We remark that in the rank one spaces the
dimension of $G/K=\dim \mathfrak{p}=m_{\alpha }+m_{2\alpha }+1$ and
it is always the case that $m_{\alpha }\geq 1+m_{2\alpha }$. For the
convenience of the reader, we list important facts about these spaces and
their restricted root systems in the appendix. Further information about
these spaces can be found in \cite{Bu}, \cite{He1}, \cite{He2}, for example.

By an \textit{orbital measure} on the compact symmetric space $G/K$, we mean
the probability measure denoted by $\mu _{z},$ for $z\in G$, defined by
\begin{equation*}
\int_{G}fd\mu _{z}=\int_{K}\int_{K}f(k_{1}zk_{2})dm_{K}(k_{1})dm_{K}(k_{2})
\end{equation*}
for all continuous functions $f$ on $G$. The orbital measure is $K$-bi-invariant and it is singular because it is supported on the double coset 
$KzK,$ a set of Haar measure zero. Since every double coset contains an
element of $A$, there is no loss of generality in assuming $z\in A$. The
measure $\mu _{z}$ is continuous (i.e., non-atomic) if and only if $z\notin N_{G}(K)$, the
normalizer of $K$ in $G$.

It was shown in \cite{GHJMAA}, that if $r=$ rank $G/K$ and $\mu _{x_{j}}$
are continuous for $j=1,\dots,2r+1$, then $\mu _{x_{1}}\ast \cdots
\ast \mu _{x_{2r+1}}$ is absolutely continuous with respect to Haar measure,
meaning its density function (or Radon-Nikodym derivative) is in $L^{1}(G)$.
In particular, the convolution product of any three continuous orbital
measures, on any rank one symmetric space, has density function in $L^{1}$. This
improved upon much earlier work of Ragozin \cite{Ra} who had shown that any
product of $\dim G/K,$ continuous, orbital measures is absolutely continuous.
If, instead, $x\in N_{G}(K)$, then $\mu _{x}^{p}$ (the $p$-fold convolution
of $\mu _{x}$) is singular with respect to Haar measure for all $p$ since in
this case $\mu _{x}^{p}$ is supported on the subset $(KxK)^{p}=x^{p}K,$ a
set of Haar measure zero.

Given $z\in A$, say $z=e^{iZ}$ for $Z\in \mathfrak{a},$ we let 
\begin{equation*}
\Phi _{z}=\{\alpha \in \Phi ^{+}:\alpha (Z)=0\text{ mod }\pi \}
\end{equation*}
be the set of \textit{annihilating roots} of $z$. The annihilating roots
are very important in questions about orbital measures and double cosets. For instance, the dimension of 
$KzK$ equals $\sum_{\beta \in \Phi ^{+}\backslash \Phi _{Z}}m_{\beta }$. It is
known that $\Phi _{z}=\Phi ^{+}$ if and only if $z\in N_{G}(K)$. In
particular, if $z\notin N_{G}(K)$, then $\dim KzK\geq m_{\alpha }$.

We call $z$ \textit{regular} if $\Phi _{z}$ is empty and then we will also
call $\mu _{z}$ regular. In this case $\dim KzK=m_{\alpha }+m_{2\alpha }$.
If $G/K$ has only one restricted root, then every $z\in A$ is either regular
or belongs to $N_{G}(K)$, equivalently, $\mu _{z}$ is either regular or not
continuous. This is the situation with $SU(2)/SO(2)$, for example. If a rank
one symmetric space $G/K$ has two positive restricted roots, then a
continuous orbital measure $\mu _{z}$ is not regular if and only if $2\alpha
(Z) \equiv 0$ mod$\pi $, but $\alpha (Z)\neq 0$ mod$\pi ,$ and then $\dim
KzK=m_{\alpha }$.

\subsection{Absolute continuity of convolution products}

In \cite{GHBAMS} it was shown that if $z_{1},z_{2}$ are both regular, then 
$\mu _{z_{1}}\ast \mu _{z_{2}}$ is absolutely continuous. Similar arguments
show the same conclusion is true if $z_{1}$ is regular and $z_{2}\notin
N_{G}(K)$ or vice versa. Our first result is to prove that the same
conclusion holds for the convolution of any two continuous orbital measures
in a rank one symmetric space.

\begin{thm}
\label{ac}
If $G/K$ is a rank one, symmetric space, then $\mu _{z_{1}}\ast \mu _{z_{2}}$
is absolutely continuous if $z_{1},z_{2}\notin N_{G}(K)$.
\end{thm}

\begin{proof}
We will write $E_{\beta }$ for any restricted root vector in $\mathfrak{g}
_{\beta }$. To simplify notation, we will write $E_{\beta }^{-}$ for 
$E_{\beta }-\theta E_{\beta }$ and $E_{\beta }^{+}$ for $E_{\beta }+\theta
E_{\beta }$. Note that $E_{\beta }^{+}\in \mathfrak{k}$ and $E_{\beta
}^{-}\in \mathfrak{p}$.

Given $z\in A$, let 
\begin{equation*}
\mathcal{N}_{z}=sp\{E_{\beta }^{-}:\text{restricted root }\beta \notin \Phi
_{z}\}\subseteq \mathfrak{p}
\end{equation*}
where $sp$ denotes the real span. It was shown in \cite{GHJMAA} that $\mu
_{z_{1}}\ast \mu _{z_{2}}$ is absolutely continuous if and only if there is
some $k\in K$ such that 
\begin{equation}
\mathfrak{p}=sp\{\mathcal{N}_{z_{1}},Ad(k)(\mathcal{N}_{z_{2}})\}.
\label{abscont}
\end{equation}

As remarked above, the result is already known if, in addition, either 
$z_{1} $ or $z_{2}$ is regular. So assume otherwise. In particular, we can
assume $G/K$ has two positive restricted roots, $\Phi _{z_{1}}=\Phi
_{z_{2}}=\{2\alpha \}$ and $\mathcal{N}_{z_{1}}=\mathcal{N}
_{z_{2}}=sp\{E_{\alpha }^{(j)-}$: $E_{\alpha }^{(j)}$ is a basis for 
$\mathfrak{g}_{\alpha }\}$. Put $E_{\alpha }^{(1)}=E_{\alpha }$ and let 
$k_{t}=\exp tE_{\alpha }^{+}\in K$ for small $t>0$.

Standard facts about root vectors and the Lie bracket implies that 
\begin{align*}
[E_{\alpha }^{+},E_{\alpha }^{(j)-}] &=[E_{\alpha },E_{\alpha
}^{(j)}]-\theta [E_{\alpha },E_{\alpha }^{(j)}]+[\theta E_{\alpha
},E_{\alpha }^{(j)}]-\theta [\theta E_{\alpha },E_{\alpha }^{(j)}] \\
&=[E_{\alpha },E_{\alpha }^{(j)}]-\theta [E_{\alpha },E_{\alpha
}^{(j)}]+H_{j},
\end{align*}
where $H_{j}\in \mathfrak{a}$.

Since $[g_{\alpha },g_{2\alpha }]=0=[g_{2\alpha },g_{2\alpha }]$ and 
$[g_{\alpha },g_{\alpha }]\subseteq g_{2\alpha }$, it follows from \cite{CDKR}
that there is a scalar $c>0$ such that for every $Z\in \mathfrak{g}_{2\alpha
}$ there is some $J_{Z}=J_{Z}(E_{\alpha })\in \mathfrak{g}_{\alpha }$ with 
\begin{equation*}
[E_{\alpha },J_{Z}]=cZ.
\end{equation*}
Temporarily fix $Z=E_{2\alpha }^{(j)}$, assume $J_{Z}=\sum d_{i}E_{\alpha
}^{(i)}$. With $J_{Z}^{-}=J_{Z}-\theta J_{Z}$, the observations above imply
\begin{equation*}
[E_{\alpha }^{+},J_{Z}^{-}]=\sum d_{i}\left( [E_{\alpha },E_{\alpha
}^{(i)}]-\theta [E_{\alpha },E_{\alpha }^{(i)}]\right) +H_{Z}
\end{equation*}
for some $H_{Z}\in \mathfrak{a}$. Thus 
\begin{equation*}
[E_{\alpha }^{+},J_{Z}^{-}]=[E_{\alpha },J_{Z}]-\theta [E_{\alpha },
J_{Z}]+H_{Z}=c(Z-\theta Z)+H_{Z}=cE_{2\alpha }^{(j)-}+H_{Z}.
\end{equation*}
Since $E_{\alpha }^{-}$, $J_{Z}^{-}\in \mathcal{N}_{z_{2}},$ and $[E_{\alpha
}^{+},E_{\alpha }^{-}]$ is a non-zero element (and hence generator) of 
$\mathfrak{a}$, it follows that 
\begin{equation*}
sp\{\mathcal{N}_{z_{1}},ad(E_{\alpha }^{+})(\mathcal{N}_{z_{2}})\}=\mathfrak{
p}\text{.}
\end{equation*}

As $\exp tE_{\alpha }^{+}=Id+t\cdot ad(E_{\alpha }^{+})$ +$P_{t}$ for some
operator $P_{t}$ with norm $O(t^{2})$, it follows that for small enough 
$t>0$, $sp\{\mathcal{N}_{z_{1}},Ad(k_{t})(\mathcal{N}_{z_{2}})\}=\mathfrak{p}$.
(We refer the reader to \cite{GHBAMS} for the details of a similar
argument.) This completes the proof.
\end{proof}

\begin{rem}
\label{Rem}
Observe that $\dim KzK=\dim \mathcal{N}_{z}$.
\end{rem}

\begin{cor}
\label{Cor:abscont}
For a rank one symmetric space $G/K$, the following are equivalent:
\begin{enumerate}
\item $\mu _{z_{1}}\ast \mu _{z_{2}}$ is absolutely continuous;
\item  $\mu _{z_{1}}$ and $\mu _{z_{2}}$ are continuous measures;
\item $\dim Kz_{1}K+\dim Kz_{2}K\geq G/K$.
\end{enumerate}
\end{cor}

\begin{proof}
Theorem \ref{ac} gives that (2) implies (1) since $\mu_z$ is continuous if and only if $z \notin N_{G}(K)$.

Since $\dim sp\{\mathcal{N}_{z_{1}},Ad(k)(\mathcal{N}_{z_{2}})\} \le \dim \mathcal{N}_{z_{1}} + \dim \mathcal{N}_{z_{2}}$,
 it is immediate from \eqref{abscont} and Remark \ref{Rem} that if
 $\dim
Kz_{1}K+\dim Kz_{2}K<\dim \mathfrak{p}=\dim G/K,$ then $\mu _{z_{1}}\ast \mu
_{z_{2}}$ is not absolutely continuous. Thus (1) implies (3).

 Lastly, we observe that if, say, $\mu_{z_1}$ is not continuous, then  $\dim Kz_{1}K=0$. Thus $\dim Kz_{1}K+\dim Kz_{2}K <\dim G/K$, so (3) implies (2).


\end{proof}

\begin{rem}
It follows from\/ {\rm \cite{Ra}} that the absolute continuity of   $\mu _{z_{1}}\ast \mu _{z_{2}}$ is also equivalent to $Kz_{1}Kz_{2}K$ having non-empty interior.
\end{rem}

\section{Convolution products that are in $L^{2}$}

In the remainder of this paper, we study when the convolution product of
orbital measures belongs to the smaller space $L^{2}(G)$. We will do this by
estimating the decay in the Fourier transform of orbital measures. For this,
we introduce further notation.

\textbf{Notation:} An irreducible, unitary representation $(\pi ,V_{\pi })$
of $G$ is called \textit{spherical} if there exists a $K$-invariant vector
in $V_{\pi }$. It is known that the dimension of the $K$-invariant subspace, 
$V_{\pi }^{K}$, is one (c.f., \cite{AGP}). We will let $X_{1}=X_{\pi
},X_{2},\dots,X_{\dim V_{\pi }}$ be an orthonormal basis for $V_{\pi },$ where
we suppose $V_{\pi }^{K}$ is spanned by $X_{\pi }$.

The following facts can essentially be found in \cite{AGP} (and are valid in
any compact symmetric space, not just those of rank one).

\begin{lemma}
For any $x,y\in G$ we have $\left\langle \widehat{\mu _{x}}(\pi )\widehat{
\mu _{y}}(\pi )X_{i},X_{j}\right\rangle =0$ if $(i,j)\neq (1,1)$ and
\begin{equation*}
\left\langle \widehat{\mu _{x}}(\pi )\widehat{\mu _{y}}(\pi )X_{\pi },X_{\pi
}\right\rangle =\left\langle \pi (x^{-1})X_{\pi },X_{\pi }\right\rangle
\left\langle \pi (y^{-1})X_{\pi },X_{\pi }\right\rangle .
\end{equation*}
\end{lemma}

\begin{proof}
It is shown in \cite{AGP} that for all $i$, $\widehat{\mu _{x}}(\pi
)X_{i}\in V_{\pi }^{K}=spX_{\pi }$, $\widehat{\mu _{x}}(\pi )X_{i}=0$ for
all $i\neq 1$ and $\left\langle \widehat{\mu _{x}}(\pi )X_{\pi },X_{\pi
}\right\rangle =\left\langle \pi (x^{-1})X_{\pi },X_{\pi }\right\rangle $.
\end{proof}

\begin{prop}
\label{propFT}For all $x,y\in G$,
\begin{align*}
\left\Vert \mu _{x}\ast \mu _{y}\right\Vert _{2}^{2} &=\left\Vert \widehat{
\mu _{x}\ast \mu _{y}}\right\Vert _{2}^{2} \\
&=\sum_{\pi \text{ \rm{spherical}}}\mathrm{\dim }V_{\pi }|\langle \pi
(x)X_{\pi },X_{\pi }\rangle \left\langle \pi (y)X_{\pi },X_{\pi
}\right\rangle |^{2}.
\end{align*}
\end{prop}

\begin{proof}
By the Peter-Weyl theorem, 
\begin{equation*}
\left\Vert \mu _{x}\ast \mu _{y}\right\Vert _{2}^{2}=\sum_{\pi \text{
spherical}}\mathrm{\dim }V_{\pi }\left\Vert \widehat{\mu _{x}\ast \mu _{y}}
(\pi )\right\Vert _{HS}^{2}.
\end{equation*}
The previous lemma and orthogonality gives 

\begin{align*}
\left\Vert \widehat{\mu _{x}\ast \mu _{y}}(\pi )\right\Vert _{HS}^{2}
&=\sum_{i=1}^{\dim V_{\pi }}\left\Vert \widehat{\mu _{x}\ast \mu _{y}}(\pi
)X_{i}\right\Vert ^{2}=\sum_{i}\sum_{j}\left\vert \left\langle \widehat{\mu
_{x}}(\pi )\widehat{\mu _{y}}(\pi )X_{i},X_{j}\right\rangle \right\vert ^{2}
\\
&=\left\vert \left\langle \widehat{\mu _{x}}(\pi )\widehat{\mu _{y}}(\pi
)X_{\pi },X_{\pi }\right\rangle \right\vert ^{2}=|\langle \pi (x)X_{\pi
},X_{\pi }\rangle \left\langle \pi (y)X_{\pi },X_{\pi }\right\rangle |^{2}.
\end{align*}

\end{proof}

We will let 
\begin{equation*}
\phi _{\pi }(x)=\langle \pi (x)X_{\pi },X_{\pi }\rangle .
\end{equation*}
These are called \textit{spherical functions} and have been well studied
(c.f., \cite[Ch. IV, V]{He}, \cite[Ch III]{He2}), particularly in the rank
one case which we will assume for the remainder of this section. The
following result is critical for us.

\begin{thm}
{\rm \cite[Ch.~V, Thm.~4.5]{He}} Let $G/K$ be a simply connected, compact
symmetric space of rank one and let $\beta $ denote the larger element in 
$\Phi ^{+}$. Let $\pi $ be a spherical representation of $G$ and let $\lambda 
$ denote the restriction of the highest weight of $\pi $ to $\mathfrak{a}$.
Then $\lambda =n\beta $ where $n$ is a positive integer. The spherical
function, $\phi _{\pi },$ is given by the hypergeometric function,
\begin{equation*}
\phi _{\pi }(x)=_{2}F_{1}\left( \frac{1}{2}m_{\beta /2}+m_{\beta }+n,-n;
\frac{1}{2}(m_{\beta /2}+m_{\beta }+1);\sin ^{2}(\beta (X)/2)\right)
\end{equation*}
where $x=\exp iX$, $X\in \mathfrak{a}$. Moreover, there is such a spherical
representation for each positive integer $n$.
\end{thm}

By $\beta $ the \textquotedblleft larger element\textquotedblright , we mean 
$\beta =2\alpha $ if there are two restricted roots and $\beta =\alpha $
otherwise. Here $m_{\beta /2}$ should be understood as $0$ if $\Phi ^{+}$
has only one element.

Using the symmetry of the first two arguments of the hypergeometric function
and the relationship between the hypergeometric functions and the Jacobi
polynomials $P_{n}^{(a,b)}(x)$, namely, 
\begin{equation*}
\frac{\Gamma (n+1)\Gamma (a+1)}{\Gamma (a+n+1)}
P_{n}^{(a,b)}(x)=_{2}F_{1}(-n,n+a+b+1,a+1;\frac{1-x}{2}),
\end{equation*}
(c.f., ~\cite{OLBC}), we obtain the following expression for the spherical
functions:

\begin{prop}
\label{main}Let $\pi _{n}$ be the spherical representation of $G$ with
highest weight restricted to $\mathfrak{a}$ equal to $n\beta $. Assume 
$z=e^{iZ}$ with $Z\in \mathfrak{a}$. Then
\begin{equation*}
\phi _{\pi _{n}}(z)=\frac{\Gamma (n+1)\Gamma (a+1)}{\Gamma (a+n+1)}
P_{n}^{(a,b)}(\cos \beta (Z))
\end{equation*}
where 
\begin{equation*}
a=\frac{1}{2}(m_{\beta /2}+m_{\beta }-1)\text{, }b=\frac{1}{2}(m_{\beta }-1).
\text{ }
\end{equation*}
\end{prop}

For the remainder of the paper, $\pi _{n}$ will denote the spherical
representation of $G$ with highest weight restricted to $\mathfrak{a}$ equal to 
$n\beta $ where $\beta =2\alpha $ if there are two positive restricted roots
and $\beta =\alpha $ otherwise.

The asymptotic dimension formula for the spherical representations $\pi _{n}$
can be derived using the Weyl dimension formula. Complicated explicit
formulas are also known, see \cite{GG}.

\begin{prop}
\label{propdim} There are constants $c_{1},c_{2}>0$ such that for any $n$
the spherical representation $\pi _{n}$ has dimension bounded by 
\begin{equation*}
c_{1}n^{m_{\alpha }+m_{2\alpha }}\leq \mathrm{\dim }V_{\pi _{n}}\leq
c_{2}n^{m_{\alpha }+m_{2\alpha }}.
\end{equation*}
\end{prop}

\begin{proof}
Suppose $\pi _{n}$ has highest weight $\lambda _{n}$ where $\lambda _{n}|_{
\mathfrak{a}}=n\beta $ with $\beta $ the largest restricted root. As shown
in \cite[Ch. V, Thm. 4.1]{He}, $\lambda _{n}$ vanishes on $\mathfrak{t\cap k}
$, thus 
\begin{equation*}
\langle \lambda _{n},\gamma \rangle =n\langle \beta ,\gamma |_{\mathfrak{a}
}\rangle .
\end{equation*}
It follows that the Weyl dimension formula implies 
\begin{equation*}
\mathrm{\dim }V_{\pi _{n}}=\prod_{\gamma \in \Sigma ^{+}}\frac{\langle
\lambda _{n}+\rho ,\gamma \rangle }{\langle \rho ,\gamma \rangle }
=\prod_{\gamma \in \Sigma ^{+}}\left( \frac{n\langle \beta ,\gamma |_{
\mathfrak{a}}\rangle }{\langle \rho ,\gamma \rangle }+1\right) .
\end{equation*}
This expression is polynomial in $n$, with each factor being either linear
or 1 depending on whether $\langle \beta ,\gamma |_{\mathfrak{a}}\rangle $
is non-zero. As $\mathfrak{a}$ is one dimensional, $\gamma |_{\mathfrak{a}
}=c\beta $ for some scalar $c$ and thus $n\langle \beta ,\gamma |_{\mathfrak{
a}}\rangle \neq 0$ if and only if $\gamma |_{\mathfrak{a}}\neq 0$. Hence the
degree of the polynomial is the number of $\gamma \in \Sigma ^{+}$ with 
$\gamma |_{\mathfrak{a}}\neq 0$, namely, $m_{\alpha }+m_{2\alpha }$.
\end{proof}

We next recall the well known asymptotic estimate for the Jacobi polynomials
which can be found in ~\cite{Sz}, for example. With these we can easily
obtain asymptotic estimates on the size of the spherical functions.

\begin{lemma}
Let $a,b\in \mathbb{R}$. Then 
\begin{equation*}
P_{n}^{(a,b)}(-1)=\binom{n+b}{n}(-1)^{n}\text{ and }P_{n}^{(a,b)}(1)=\binom{
n+a}{n}\text{,}
\end{equation*}
while if $\theta \in (0,\pi )$, 
\begin{equation*}
P_{n}^{(a,b)}(cos\theta )=k(\theta )n^{-\frac{1}{2}}cos(N\theta +\gamma
)+O(n^{-\frac{3}{2}})
\end{equation*}
where 
\begin{equation*}
N=n+\frac{a+b+1}{2}\text{, }\gamma =-\frac{\pi }{2}(a+\frac{1}{2})
\end{equation*}
and 
\begin{equation*}
k(\theta )=\pi ^{-\frac{1}{2}}(sin(\frac{\theta }{2}))^{-a-\frac{1}{2}}(cos(
\frac{\theta }{2}))^{-b-\frac{1}{2}}>0.
\end{equation*}
\end{lemma}

\begin{cor}
\label{sizesph} {\rm (a)} If $z\in A$ is regular, then
\begin{equation*}
\left\vert \phi _{\pi _{n}}(z)\right\vert \leq Cn^{-\frac{1}{2}(m_{\alpha
}+m_{2\alpha })}.
\end{equation*}
{\rm (}Here $m_{2\alpha }$ should be understood as $0$ in the single root case.{\rm )}

{\rm (b)} If $z\in A\diagdown N_{G}(K),$ but is not regular, then there are
positive constants $C_{1},C_{2}$ such that 
\begin{equation*}
C_{1}n^{-\frac{1}{2}m_{\alpha }}\leq \left\vert \phi _{\pi
_{n}}(z)\right\vert \leq C_{2}n^{-\frac{1}{2}m_{\alpha }}.
\end{equation*}
\end{cor}

\begin{proof}
These estimates follow directly from the previous result since the Gamma
function is known to satisfy 
\begin{equation*}
\frac{\Gamma (n+1)}{\Gamma (a+n+1)}=O(n^{-a})\text{ and }\binom{n+s}{n}
=O(n^{s})\text{ for }s>0.
\end{equation*}
\end{proof}

We are now ready to prove our main result. Note that throughout the proof $C$
will denote a constant that can vary from one line to the next.

\begin{thm}
\label{Th:main}Assume $G/K$ is a rank one, simple, simply connected,
compact, symmetric space that is not isomorphic to $SU(2)/SO(2)$. Assume 
$z_{1},z_{2}\in A\diagdown N_{G}(K).$

{\rm (a)} If either of $z_{1}$ or $z_{2}$ is regular, then $\mu _{z_{1}}\ast \mu
_{z_{2}}\in L^{2}(G)$.

{\rm (b)} If $G/K$ is not type $AIII$ or $CII$ with $q=2,$ or type $FII$, then 
$\mu _{z_{1}}\ast \mu _{z_{2}}\in L^{2}(G)$.

{\rm (c)} If $G/K$ is type $AIII$ or $CII$ with $q=2,$ or type $FII,$ and neither 
$z_{1}$ nor $z_{2}$ is regular, then $\mu _{z_{1}}\ast \mu _{z_{2}}\notin
L^{2}(G)$.
\end{thm}

\begin{rem}
The compact symmetric spaces of type $AIII$ or $CII$ with $q=2$ are those
isomorphic to $SU(3)/(S(U(2)\times U(1)))$ or $Sp(6)/(Sp(4)\times Sp(2)),$
and those of type $FII$ are $F_{4}/SO(9)$. The significance of these is that
they are the rank one spaces with two positive restricted roots that satisfy 
$m_{\alpha }=m_{2\alpha }+1.$
\end{rem}

\begin{proof}
(a) From Prop.~\ref{propFT} we have
\begin{equation*}
\left\Vert \mu _{z_{1}}\ast \mu _{z_{2}}\right\Vert _{2}^{2}=\sum_{n}\dim
V_{\pi _{n}}\left\vert \phi _{\pi _{n}}(z_{1})\phi _{\pi
_{n}}(z_{2})\right\vert ^{2}.
\end{equation*}
If both $z_{1}$ and $z_{2}$ are regular, then Cor.~\ref{sizesph}(i) gives 
\begin{equation*}
\left\vert \phi _{\pi _{n}}(z_{j})\right\vert \leq Cn^{^{-}\frac{1}{2}
(m_{\alpha }+m_{2\alpha })}
\end{equation*}
for both $j=1,2$. Combining this with the fact that $\dim V_{\pi _{n}}\leq
O(n^{m_{\alpha }+m_{2\alpha }})$ yields the bound 
\begin{equation*}
\left\Vert \mu _{z_{1}}\ast \mu _{z_{2}}\right\Vert _{2}^{2}\leq C\sum_{n}
\mathrm{\dim }V_{\pi _{n}}n^{-2(m_{\alpha }+m_{2\alpha })}\leq
C\sum_{n}n^{-(m_{\alpha }+m_{2\alpha })}\text{.}
\end{equation*}
For all rank one symmetric spaces other than $SU(2)/SO(2)$, $m_{\alpha
}+m_{2\alpha }\geq 2$ (see the appendix) and hence this sum converges. Thus 
$\mu _{z_{1}}\ast \mu _{z_{2}}\in L^{2}$.

Next, suppose $z_1$, but not $z_2$, is regular. Then Cor.~\ref{sizesph}(ii) gives $\left\vert
\phi _{\pi _{n}}(z_{2})\right\vert \leq Cn^{-m_{\alpha }/2}$ and similar arguments to the first case shows that
\begin{equation*}
\left\Vert \mu _{z_{1}}\ast \mu _{z_{2}}\right\Vert _{2}^{2}\leq C\sum_{n}
\mathrm{\dim }V_{\pi _{n}}n^{-(m_{\alpha }+m_{2\alpha })}n^{-m_{\alpha
}}\leq C\sum_{n}n^{-m_{\alpha }}.
\end{equation*}
This sum is finite since $m_{\alpha }\geq 2$ whenever there are two positive
restricted roots, as must be the case if there is such a element $z_{2}$.

(b) We can assume neither $z_{1},z_{2}$ are regular for otherwise we simply
apply (a). Arguing as above gives 
\begin{equation*}
\left\Vert \mu _{z_{1}}\ast \mu _{z_{2}}\right\Vert _{2}^{2}\leq
C\sum_{n}n^{m_{\alpha }+m_{2\alpha }}n^{-2m_{\alpha }}\leq
C\sum_{n}n^{-m_{\alpha }+m_{2\alpha }}.
\end{equation*}
But $m_{\alpha }-m_{2\alpha }\geq 2$ in all the rank one symmetric spaces
other than those we have listed.

(c) If neither $z_{1},z_{2}$ are regular, then the lower bounds of 
Prop.~\ref{propdim} and Cor.~\ref{sizesph}(b) give 
\begin{equation*}
\left\Vert \mu _{z_{1}}\ast \mu _{z_{2}}\right\Vert _{2}^{2}\geq C^{\prime
}\sum_{n}n^{-m_{\alpha }+m_{2\alpha }}
\end{equation*}
for some constant $C^{\prime }>0$. Thus $\mu _{z_{1}}\ast \mu _{z_{2}}\notin
L^{2}$ whenever $m_{\alpha }-m_{2\alpha }=1$, as is the case for these
symmetric spaces.
\end{proof}

The final result of this section is to show that any three-fold convolution
of continuous orbital measures is in $L^{2}$.

\begin{prop}
If $G/K$ is any rank one compact symmetric space and $z_{j}\notin
N_{G}(K) $, $j=1,2,3$, then $\mu _{z_{1}}\ast \mu _{z_{2}}\ast \mu
_{z_{3}}\in L^{2}$.
\end{prop}

\begin{proof}
First, assume $G/K$ is not $SU(2)/SO(2)$. By (a) of the previous theorem we
can assume none of $z_{j}$, $j=1,2,3$, are regular. Then, as above, we have
\begin{equation*}
\left\Vert \mu _{z_{1}}\ast \mu _{z_{2}}\ast \mu _{z_{3}}\right\Vert
_{2}^{2}\leq C\sum_{n}n^{m_{\alpha }+m_{2\alpha }}n^{-3m_{\alpha }}\leq
C\sum_{n}n^{-2m_{\alpha }+m_{2\alpha }}
\end{equation*}
and this is finite for all these symmetric spaces.

A proof of this result for $SU(2)/SO(2)$ is given in \cite{AGP}. It can also
be shown by a similar argument to the above, noting that all $z_{j}\notin
N_{SU(2)}(SO(2))$ are regular and using the asymptotic formula from 
Cor.~\ref{sizesph}(a).
\end{proof}

\section{Convolution products in $SU(2)/SO(2)$}

In this section we will show that \textit{no} product of two orbital
measures in $SU(2)/SO(2)$ has an $L^{2}$ density function. We will make use
of the following well known facts.

\begin{lemma}
\label{FS}Consider the trigonometric series for $x\in [ 0,\pi ]:$
\begin{equation*}
\text{ }\sum_{n=1}^{\infty }\frac{\sin nx}{n}\text{ and }\sum_{n=1}^{\infty }
\frac{\cos nx}{n}.
\end{equation*}
The first converges pointwise to the odd, $2\pi $-periodic extension of 
$(\pi -x)/2$. The second converges pointwise to the even, $2\pi $-periodic
extension of $-\log (2\sin (x/2))$ except at $x=0$.
\end{lemma}

We note that every double coset of $SU(2)/SO(2)$ contains an element $z\in A$
of the form $
\begin{bmatrix}
e^{i\theta } & 0 \\ 
0 & e^{-i\theta }
\end{bmatrix}$ where $\theta \in $ $[0,\pi /2]$. We will abuse notation and let $z$ also
denote the angle $\theta $. There is only one restricted root, $\alpha ,$ of
multiplicity $m_{\alpha }=1$ and $\alpha (z)=2z$.

\begin{thm}
\label{SU}If $z_{1},z_{2}\in SU(2)$, then $\mu _{z_{1}}\ast \mu
_{z_{2}}\notin L^{2}$.
\end{thm}

\begin{proof}
Throughout the proof $C$ will denote a (strictly) positive constant that can
change from line to another.

Without loss of generality we can assume $z_{j}\in A$. Furthermore, we can
assume both (the angles) $z_{j}\in (0,\pi /2)$ because if $z_{j}=0$ or $\pi
/2$, then $\alpha (z_{j})=2z_{j}=0$ mod$\pi $ so $z_{j}\in N_{G}(K)$ and in
this case even $\mu _{z_{1}}\ast \mu _{z_{2}}\notin L^{1}$.

The spherical representation $\pi _{n}$ of $SU(n)$, with highest weight
restricted to $\mathfrak{a}$ equal to $n\alpha $ is known to have dimension 
$2n+1$ (c.f. \cite{AGP} or \cite[p. 322]{He2}). From Prop.~\ref{main}
we have 
\begin{equation*}
\phi _{\pi _{n}}(z)=P_{n}^{(0,0)}(\cos \alpha (z))=P_{n}^{(0,0)}(\cos 2z).
\end{equation*}
The asymptotic estimates for Jacobi polynomials give
\begin{equation*}
\phi _{\pi _{n}}(z)=Cn^{-1/2}\left( \cos ((n+1/2)2z-\pi /4)+O(n^{-1})\right)
.
\end{equation*}
Squaring gives 
\begin{align*}
\left( \phi _{\pi _{n}}(z)\right) ^{2} &=Cn^{-1}\left( \cos
^{2}((2n+1)z-\pi /4)+O(n^{-1}\right)  \\
&=Cn^{-1}\left( \frac{\cos (2z(2n+1)-\pi /2)+1}{2}+O(n^{-1})\right) \\
&=Cn^{-1}\left( \sin (2z(2n+1))+1+O(n^{-1})\right) .
\end{align*}
Hence
\begin{equation*}
\left\Vert \mu _{z_{1}}\ast \mu _{z_{2}}\right\Vert _{2}^{2}=\sum_{n}\mathrm{
\dim }V_{\pi _{n}}\left( \phi _{\pi _{2n}}(z_{1})\right) ^{2}\left( \phi
_{\pi _{2m}}(z_{2})\right) ^{2}
\end{equation*}
\begin{align*}
&=\sum_{n}C\frac{(2n+1)}{n^{2}}\left( \sin
(2z_{1}(2n+1))+1+O(n^{-1})\right) \left( \sin
(2z_{2}(2n+1))+1+O(n^{-1})\right) \\
&=\sum_{n}\frac{C}{n}\left( \left( \sin (2z_{1}(2n+1))+1\right) \left( \sin
(2z_{2}(2n+1))+1\right) +O(n^{-1})\right) .
\end{align*}

We claim this sum diverges. Of course, the convergence or divergence of the
sum depends only on the convergence or divergence of 
\begin{equation*}
\sum_{n}\frac{1}{n}\left( \sin (2z_{1}(2n+1))+1\right) (\sin
(2z_{2}(2n+1))+1)
\end{equation*}
and therefore upon the sum 
\begin{equation}
\sum_{n}\frac{1}{n}\left(\sin (2z_{1}(2n+1))\sin (2z_{2}(2n+1))+\sin
(2z_{1}(2n+1))+\sin (2z_{2}(2n+1))+1\right).  \label{sum}
\end{equation}
As $\sin ((2n+1)\theta )=\sin 2n\theta \cos \theta +\sin \theta \cos
2n\theta $, Lemma~\ref{FS} implies 
\begin{equation*}
\sum_{n}\frac{\sin ((2n+1)\theta )}{n}
\end{equation*}
converges for any $\theta =2z_{1},2z_{2}$ as $4z_{j}\neq 0$ mod$2\pi $. Thus
\eqref{sum} converges if and only if 
\begin{equation}
\sum_{n}\frac{1}{n}(\sin (2z_{1}(2n+1))\sin (2z_{2}(2n+1))+1)<\infty .
\label{sum1}
\end{equation}

Another application of basic trigonometric identities shows there are
scalars $c_{j}=c_{j}(z_{1},z_{2})$ such that 
\begin{align*}
\sin (2z_{1}(2n+1))\sin (2z_{2}(2n+1)) &=\cos (4n(z_{1}-z_{2}))c_{1}-\sin
(4n(z_{1}-z_{2}))c_{2} \\
&-\cos (4n(z_{1}+z_{2}))c_{3}+\sin (4n(z_{1}+z_{2}))c_{4}.
\end{align*}
It follows from the lemma above that
\begin{equation*}
\sum_{n}\frac{1}{n}\sin (2z_{1}(2n+1))\sin (2z_{2}(2n+1))
\end{equation*}
converges if $4(z_{1}\pm z_{2})\neq 0$ mod$2\pi $. Of course, in this case 
\eqref{sum1} diverges.

It remains to consider the two possibilities $z_1 \pm z_2 \equiv 0$ mod$(\pi/2)$.

Case 1: $z_{1}-z_{2}\equiv 0$ mod$\pi /2$. As $z_{1},z_{2}\in (0,\pi /2)$,
this can only happen if $z_{1}=z_{2}$. Then
\begin{equation*}
1+\sin (2z_{1}(2n+1))\sin (2z_{2}(2n+1))=1+\sin ^{2}(2z_{1}(2n+1))
\end{equation*}
\begin{equation*}
=\frac{1}{2}(3-(\cos 8nz_{1}\cos 4z_{1}-\sin 8nz_{1}\sin 4z_{1})).
\end{equation*}
If $z_{1}\neq \pi /4$, then $\sum \left( \cos 8nz_{1}\right) /n$ converges
and hence \eqref{sum1} diverges. If $z_{1}=z_{2}=\pi /4$, then direct
substitution shows \eqref{sum1} diverges.

Case 2: $z_{1}+z_{2}\equiv 0$ mod$\pi /2$. Then $z_{1}+z_{2}=\pi /2$. In
this case, $\sin (2z_{2}(2n+1))=\sin (2z_{1}(2n+1))$ and hence the arguments
are the same.
\end{proof}

To conclude, we summarize our results. We will say $G/K$ satisfies the 
$L^{1}\longleftrightarrow L^{2}$ dichotomy if $\mu _{x}\ast \mu _{y}\in L^{1}$
implies $\mu _{x}\ast \mu _{y}\in L^{2}$. The following is an immediate consequence of Theorems~\ref{Th:main} and \ref{SU}.

\begin{cor}
The rank one symmetric space $G/K$ satisfies the $L^{1}\longleftrightarrow
L^{2}$ dichotomy if and only if $m_{\alpha }-m_{2\alpha }>1$. More
specifically:

{\rm (i)} When $G/K = SU(2)/SO(2)$ ($m_{\alpha }=1$, $m_{2\alpha }=0$), then 
$\mu _{x}\ast \mu _{y}$ $\notin L^{2}$ for any $x,y$.

{\rm (ii)} If $G/K$ is type $AIII$ or $CII$ with $q=2,$ or type $FII$, {\rm (}the other
symmetric spaces with $m_{\alpha }=m_{2\alpha }+1${\rm )} and $\mu _{x}\ast \mu
_{y}\in L^{1}$, then $\mu _{x}\ast \mu _{y}\notin L^{2}$ if and only if neither 
$x$ nor $y$ is regular.
\end{cor}

\begin{cor}
If $G/K$ is a rank one symmetric space, then $\mu _{z_{1}}\ast \mu
_{z_{2}}\in L^{2}$ if and only if $\dim Kz_{1}K+\dim Kz_{2}K>G/K$.
\end{cor}

\begin{proof}
First, suppose $G/K$ is not type $AIII$ or $CII$ with $q=2$, type $FII$ or
(isomorphic to) $SU(2)/SO(2)$. Then Theorem~\ref{Th:main} says $\mu
_{z_{1}}\ast \mu _{z_{2}}\in L^{2}$ if and only if $\mu _{z_{1}},\mu
_{z_{2}} $ are continuous. In this case, $\dim Kz_{j}K\geq m_{\alpha }$ and as 
$m_{\alpha }\geq 2+m_{2\alpha }$ for these symmetric spaces, it follows that $\dim
Kz_{1}K+\dim Kz_{2}K\geq m_{\alpha }+2+m_{2\alpha }>\dim G/K$. 

If, instead,
$\dim Kz_{1}K+\dim Kz_{2}K\leq \dim G/K$, then $\dim Kz_{j}K<m_{\alpha }$ for
some $j$. But that means $\mu _{z_{j}}$ is not continuous, hence $\mu
_{z_{1}}\ast \mu _{z_{2}}$ is not even in $L^{1}$.

If $G/K$ is type $AIII$ or $CII$ with $q=2$ or type $FII,$ then $\mu
_{z_{1}}\ast \mu _{z_{2}}\in L^{2}$ if and only if both $\mu _{z_{1}},\mu
_{z_{2}}$ are continuous and at least one is regular. But then $\dim
Kz_{1}K+\dim Kz_{2}K\geq $ $m_{\alpha }+m_{\alpha }+m_{2\alpha }>\dim G/K$
as $m_{\alpha }\geq 2$. On the other hand, if neither $\mu _{z_{1}}$ or $\mu
_{z_{2}}$ is regular, then $\dim Kz_{1}K+\dim Kz_{2}K\leq 2m_{\alpha
}=m_{\alpha }+1+m_{2\alpha }=\dim G/K.$

Finally, if $G/K$ is isomorphic to $SU(2)/SO(2)$, then $\mu _{z_{1}}\ast \mu _{z_{2}}\notin
L^{2}$ for any $z_{1},z_{2}$, while $\dim Kz_{1}K+\dim Kz_{2}K\leq 2=\dim
G/K $ for all $z_{1},z_{2}$.
\end{proof}

{\bf{Acknowledgement}}: We thank F. Ricci for helpful conversations.

\section{Appendix}

We list here the families of compact symmetric spaces of rank one, along
with the multiplicities of $\alpha ,2\alpha $ and restricted root system 
$\Phi ^{+}$. These facts can be found in \cite[Ch. X]{He1}.
\begin{table}[h]
\centering
\begin{tabular}{ccccc}
\toprule
\text{Type} & $G/K$ & $\Phi ^{+}$ & $m_{\alpha }$ & $m_{2\alpha }$ \\ 
\midrule
$AI$ & $SU(2)/SO(2)$ & $A_{1}$ & $1$ & $-$ \\ 
$AII$ & $SU(4)/Sp(4)$ & $A_{1}$ & $4$ & $-$ \\ 
$AIII$ & 
$\begin{array}{c}
SU(q+1)/S(U(q)\times U(1)) \\ 
q>1
\end{array}$
& $BC_{1}$ & $2(q-1)$ & $1$ \\ 
$BII$ & 
$\begin{array}{c}
SO(q+1)/S(O(q)\times O(1)) \\ 
q>2
\end{array}$
& $A_{1}$ & $q-1$ & $-$ \\ 
$CII$ & 
$\begin{array}{c}
Sp(2q+2)/Sp(2q)\times Sp(2) \\ 
q>1
\end{array}$
& $BC_{1}$ & $4(q-1)$ & $3$ \\ 
$FII$ & $F_{4}/SO(9)$ & $BC_{1}$ & $8$ & $7$ \\
\bottomrule
\end{tabular}
\end{table}
We have excluded $BII$ with $q=2$ as this is isomorphic to $SU(2)/SO(2)$.
Similarly, the only simple, rank one symmetric space of type $DIII$ is
isomorphic to $SU(4)/S(U(3)\times U(1))$, i.e., type $AIII$ with $q=3$.

\end{document}